\documentclass{amsart}

\usepackage{amscd}
\usepackage{amssymb}

\usepackage[T1]{fontenc}
\usepackage[latin1]{inputenc}
\usepackage{hyperref}
\usepackage[arrow,curve,matrix]{xy}
\usepackage{times}
\usepackage{graphicx}
\usepackage{color}
\usepackage{flafter}
\usepackage{placeins}
\usepackage{enumitem}

\def\PP{{ \mathbf{P}}}

\def\OO{{\mathcal O}}
\def\R{\mathbf{R}}

\def\D{\mathbf{D}}

\def\F{\mathcal{F}}
\def\E{\mathcal{E}}

\def\Pic{{\rm Pic}}

\theoremstyle{plain}

\newtheorem{theorem}{Theorem}[section]

\newtheorem{proposition/example}[theorem]{Proposition/Example}

\newtheorem{corollary}[theorem]{Corollary}

\newtheorem{problem}[theorem]{Problem}
\newtheorem{variant}[theorem]{Variant}
\newtheorem{conjecture}[theorem]{Conjecture}

\theoremstyle{definition}

\newtheorem{conjecture/question}[theorem]{Conjecture/Question}

\newtheorem{remark/definition}[theorem]{Remark/Definition}
\newtheorem{definition/notation}[theorem]{Definition/Notation}

 \theoremstyle{remark}

\numberwithin{equation}{section}


\begin{document}

\title[Derived equivalence and non-vanishing loci]{Derived 
equivalence and non-vanishing loci}

\author{Mihnea Popa}
\address{Department of Mathematics, University of Illinois at Chicago,
851 S. Morgan Street, Chicago, IL 60607, USA } \email{{\tt
mpopa@math.uic.edu}}
\thanks{The author was partially supported by the NSF grant DMS-1101323.}



\setlength{\parskip}{.09 in}

\dedicatory{To Joe Harris, with great admiration.}

\maketitle

\section{The conjecture and its variants}

The purpose of this note is to propose and motivate a conjecture on the behavior of cohomological support loci for topologically trivial line bundles under derived equivalence, to verify it in the case of surfaces, and to explain further developments.
The reason for such a conjecture is the desire to understand the relationship between the cohomology groups of (twists of) the canonical line bundles of  derived equivalent varieties. This in turn is motivated by the following well-known problem, stemming from a prediction of Kontsevich in the case of Calabi-Yau manifolds (and which would also follow from the main conjecture in Orlov \cite{orlov2}).

\begin{problem}
Let $X$ and $Y$ be smooth projective complex varieties with $\D(X) \simeq \D(Y)$.  Is it true that
$h^{p,q}(X) = h^{p,q}(Y)$ for all $p$ and  $q$?
\end{problem}

Here, given a smooth projective complex variety $X$, we denote by $\D(X)$ the bounded derived 
category of coherent sheaves $\D^{{\rm b}} ({\rm Coh}(X))$.  
For surfaces the answer is yes, for instance because of the derived invariance of Hochschild homology \cite{orlov1}, \cite{caldararu}. 
This is also true for threefolds, again using the invariance of Hochschild homology, together 
with the behavior of the Picard variety under derived equivalence \cite{PS}. In general, even the 
invariance of $h^{0, q}$ with $1< q < \dim X$ is not known at the moment, and this leads to the 
search for possible methods for circumventing the difficult direct study of the cohomology groups $H^i (X, \omega_X)$.

More precisely, in \cite{PS} it is shown that if $\D(X) \simeq \D(Y)$, then ${\rm Pic}^0 (X)$ and ${\rm Pic}^0 (Y)$ are isogenous. 
This opens the door towards studying the behavior or more refined objects associated to irregular varieties
(i.e. those with $q (X) = h^0 (X, \Omega_X^1) > 0$) under derived equivalence. Among the most important such objects 
are the \emph{cohomological support loci} of the canonical bundle: given a 
smooth projective $X$, for $i = 0, \ldots, \dim X$ one defines
$$V^i( \omega_X) : = \{ \alpha ~|~ H^i (X, \omega_X \otimes \alpha) \neq 0\} \subseteq {\rm Pic}^0 (X).$$
By semicontinuity, these are closed algebraic subsets of ${\rm Pic}^0 (X)$. 
It has become clear in recent years that these loci are the foremost tool in studying the special birational geometry of irregular 
varieties, with applications ranging from results about singularities of theta divisors 
\cite{EL} to the proof of Ueno's conjecture \cite{ChH}.
They are governed by the following fundamental results of generic vanishing theory (\cite{GL1}, \cite{GL2}, \cite{Arapura}, \cite{Simpson}):

\noindent
$\bullet$~ \,\,~If $a: X\rightarrow {\rm Alb}(X)$ is the Albanese map of $X$, then 
\begin{equation}\label{gv}
{\rm codim}~V^i (\omega_X) \ge i - \dim X + \dim a(X) \,\,\,\, {\rm for~all~i},
\end{equation}
and there exists an $i$ for which this is an equality.

\noindent
$\bullet$~\,\,~The  irreducible components of each $V^i (\omega_X)$ are torsion translates of
abelian subvarieties of $\Pic^0(X)$.

\noindent
$\bullet$ ~\,\,~Each positive dimensional component of some $V^i (\omega_X)$ corresponds to a fibration $f:X\rightarrow Y$ 
onto a normal variety with $0 < \dim Y \le \dim X - i$ and with generically finite Albanese map. 

The main point of this note is the following conjecture, saying that cohomological support loci should be preserved by derived equivalence. 
In the next sections I will explain that the conjecture holds for surfaces, and 
that is almost known to hold for threefolds.

\begin{conjecture}\label{main}
Let $X$ and $Y$ be smooth projective varieties with $\D(X) \simeq \D(Y)$. Then 
$$V^i (\omega_X)  \simeq V^i (\omega_Y) \,\,\,\, {\rm for ~all~} i \ge 0.$$ 
\end{conjecture}

Note that I am proposing isomorphism, even though the ambient spaces ${\rm Pic}^0 (X)$ and ${\rm Pic}^0 (Y)$ may only be isogenous.
There are roughly speaking three main reasons for this: (1) the conjecture is known to hold for surfaces and for most threefolds, 
as explained in \S2 and \S3; (2) it holds for $V^0$ in arbitrary dimension, as explained at the beginning of \S3; (3) more heuristically, 
according to \cite{PS} the failure of isomorphism at the level of ${\rm Pic}^0$ is induced by the presence of abelian varieties in the picture, 
and for these all cohomological support loci consist only of the origin.

Furthermore, denote by $V^i( \omega_X)_0$ the union of the 
irreducible components of $V^i (\omega_X)$  passing through the origin. Generic vanishing theory tells 
us that in many applications one only needs to control well $V^i (\omega_X)_0$. 
In fact, for all applications I currently have in mind, the following variant of Conjecture \ref{main} suffices.

\begin{variant}\label{var1}
Under the same hypothesis, 
$$V^i (\omega_X)_0  \simeq V^i (\omega_Y)_0 \,\,\,\, {\rm for ~all~} i \ge 0.$$ 
\end{variant}

The key fact implied by this variant is that, excepting perhaps surjective maps to abelian varieties, 
roughly speaking two derived equivalent varieties must have the same types of fibrations onto 
lower dimensional irregular varieties (see Corollary \ref{fibrations}). This would hopefully allow for 
further geometric tools in the classification of irregular derived partners.
Even weaker versions of Conjecture \ref{main} and Variant \ref{var1} are of interest, as they are all that is needed in other applications. 

\begin{variant}\label{var2}
Under the same hypothesis, 
$$\dim V^i (\omega_X)  = \dim V^i (\omega_Y) \,\,\,\, {\rm for ~all~} i \ge 0.$$ 
\end{variant}

\begin{variant}\label{var3}
Under the same hypothesis, 
$$\dim V^i (\omega_X)_0  = \dim V^i (\omega_Y)_0 \,\,\,\, {\rm for ~all~} i \ge 0.$$ 
\end{variant}

For instance, Variant \ref{var3} implies the derived invariance of the Albanese dimension (see Corollary \ref{albanese_dimension}). Other numerical applications, and progress due to Lombardi \cite{lombardi} in the case of $V^0$ and $V^1$, and on the full conjecture for threefolds, are described in \S3. In \S2 I present a proof of Conjecture \ref{main} in the case of surfaces.

\section{A proof of Conjecture \ref{main} for surfaces}

Due to the classification of Fourier-Mukai equivalences between surfaces, in this case the conjecture reduces to a calculation of all possible cohomological support loci via a case by case analysis, combined with some elliptic surface theory, generic vanishing theory, well-known results of Koll\'ar on higher direct images of dualizing sheaves, and of Beauville on the positive dimensional components of $V^1 (\omega_X)$. This of course does not have much chance to generalize to higher dimensions. In the next section I will point to more refined techniques developed by Lombardi \cite{lombardi}, which recover the case of surfaces, but do address higher 
dimensions as well.

\begin{theorem}\label{surfaces}
Conjecture \ref{main} holds when $X$ and $Y$ are smooth projective surfaces.
\end{theorem}
\begin{proof}
The first thing to note is that, due to the work of Bridgeland-Maciocia \cite{bm} and Kawamata \cite{kawamata}, 
Fourier-Mukai equivalences of surfaces are completely classified. According to \cite{kawamata} Theorem 1.6, the only non-minimal 
surfaces that can have derived partners are rational elliptic, and therefore regular. Hence we can restrict to minimal surfaces. Among these on the other hand, 
according to \cite{bm} Theorem 1.1, only abelian, $K3$ and elliptic surfaces can have distinct derived partners. 

Now $K3$ surfaces are again regular, hence for these the problem is trivial. On the other hand, on any abelian variety $A$ (of arbitrary dimension) one has 
$$V^i (\omega_A) = \{ 0\} \,\,\,\, {\rm for ~all~} i,$$
and since the only derived partners of abelian varieties are again abelian varieties (see \cite{hn} Proposition 3.1; cf. also \cite{PS}, end of \S3), the problem is again trivial. Therefore our question is truly a question about elliptic surfaces which are not rational. Moreover, according to \cite{bm}, bielliptic surfaces do not have non-trivial derived partners. We are 
left with certain elliptic fibrations over $\PP^1$, and with elliptic fibrations over smooth projective curves of genus at least $2$. I will try in each case to present the most
elementary proof I am aware of.

Let first $f: X \rightarrow \PP^1$ be an elliptic surface over $\PP^1$. Since our problem is non-trivial only for irregular surfaces, requiring $q (X) \neq 0$ 
we must then have that $q (X) = 1$, which implies that $f$ is 
isotrivial, and in fact that $X$ is a $\PP^1$-bundle 
$$\pi: X \longrightarrow E$$
over an elliptic curve $E$. We can now compute the cohomological support loci $V^i (\omega_X)$ explicitly.
Note first that $V^2(\omega_X) = \{0\}$ for any smooth projective surface, by Serre duality. In the case at 
hand, note also that $\pi$ is the Albanese map of $X$. Therefore, identifying line bundles in 
${\rm Pic}^0 (X)$ and ${\rm Pic}^0 (E)$, for every $\alpha \in {\rm Pic}^0 (X)$ we have
$$H^0 (X, \omega_X \otimes \alpha) \simeq H^0 (E, \pi_* \omega_X \otimes \alpha),$$
which implies that $V^0 (\omega_X) = V^0 (E, \pi_* \omega_X)$.\footnote{In general, for any coherent sheaf $\F$ on a smooth projective variety $Z$, and any integer $i$, we denote
$V^i (Z, \F) : = \{ \alpha \in {\rm Pic}^0 (Z) ~|~ h^i (Z, \F \otimes \alpha) \neq 0 \}$.} But given that $\pi_* \omega_X$ must be torsion-free, and the fibers of $\pi$ are rational curves, we have $\pi_* \omega_X = 0$, and so 
$V^0 (\omega_X) = \emptyset$. We are left with computing $V^1 (\omega_X)$. For this, recall that 
by \cite{kollar2} Theorem 3.1, in $\D (E)$ we have the decomposition
$$\R \pi_* \omega_X \simeq \pi_* \omega_X \oplus R^1\pi_* \omega_X [-1] \simeq 
R^1\pi_* \omega_X [-1],$$
where the second isomorphism follows from what we said above.
Therefore, for any $\alpha \in {\rm Pic}^0 (X)$, we have 
$$H^1 ( X, \omega_X \otimes \alpha) \simeq H^0 (E, R^1 \pi_* \omega_X \otimes \alpha).$$
Finally, \cite{kollar1} Proposition 7.6 implies that, as the top non-vanishing higher direct image, 
$$R^1 \pi_* \omega_X \simeq \omega_E \simeq \OO_E,$$
which immediately gives that $V^1 (\omega_X) = \{0\}$. In conclusion, we have obtained that for the 
type of surface under discussion we have 
$$V^0 (\omega_X) = \emptyset, \,\,\,\, V^1 (\omega_X) = \{0\}, \,\,\,\, V^2 (\omega_X) = \{0\}.$$
Finally, if $Y$ is another smooth projective surface such that $\D(X) \simeq \D(Y)$, then due to \cite{bm} 
Proposition 4.4 we have that $Y$ is another elliptic surface over $\PP^1$ with the same properties as $X$, which leads therefore to the same cohomological support loci.

Assume now that $f: X \rightarrow C$ is an elliptic surface over a smooth projective curve $C$ of 
genus $g \ge 2$ (so that $\kappa(X) = 1$).  By the same \cite{bm} Proposition 4.4, if $Y$ is another smooth projective surface such that $\D(Y) \simeq \D(X)$, then $Y$ has an elliptic fibration structure $h: Y \rightarrow C$ over the same curve (and with isomorphic fibers over a Zariski open set in $C$; in fact it is a relative Picard 
scheme associated to $f$). There are two cases, namely when $f$ is isotrivial, and  when it is not. It is well known (see e.g. \cite{beauville1} Exercise IX.1 and \cite{friedman} Ch.7) that $f$ is isotrivial if and only if $q (X) = g + 1$ (in which case the only singular fibers are multiple fibers with smooth reduction), and it is not isotrivial if and only if $q (X) = g$. Since we know that $q (X) = q (Y)$, we conclude that $h$ must be of the same type as $f$. We will again compute all $V^i (\omega_X)$ in the two cases.

Let's assume first that $f$ is not isotrivial.  As mentioned above, in this case $q (X) = g$, and in fact 
$f^*: {\rm Pic}^0 (C) \rightarrow {\rm Pic}^0 (X)$ is an isomorphism. To compute $V^1 (\omega_X)$, we use 
again \cite{kollar2} Theorem 3.1, saying that in $\D(C)$ there is a direct sum decomposition
$$\R f_* \omega_X \simeq f_* \omega_X \oplus R^1 f_* \omega_X [-1],$$
and therefore for each $\alpha \in {\rm Pic}^0 (X)$ one has
$$H^1 (X, \omega_X \otimes \alpha) \simeq H^1 (C, f_* \omega_X \otimes \alpha) \oplus
H^0 (C, R^1 f_* \omega_X \otimes \alpha).$$
Once again using \cite{kollar1} Proposition 7.6, we also have $R^1 f_* \omega_X \simeq \omega_C$. This,  combined with the decomposition above, gives the inclusion
$$f^*: {\rm Pic}^0 (C) = V^0 (C, \omega_C)  \hookrightarrow V^1 (\omega_X),$$
finally implying $V^1 (\omega_X) = {\rm Pic}^0 (X)$.
Finally, note that by the Castelnuovo inequality \cite{beauville1} Theorem X.4, we have $\chi (\omega_X) \ge 0$. Now the Euler characteristic is a 
deformation invariant, hence $\chi (\omega_X \otimes \alpha) \ge 0$ for all $\alpha \in {\rm Pic}^0 (X)$.
For $\alpha \neq \OO_X$, this gives 
$$h^0 (X, \omega_X \otimes \alpha) \ge h^1 (X, \omega_X \otimes \alpha),$$ 
so that $V^1 (\omega_X) \subset V^0 (\omega_X)$. By the above, we obtain 
$V^0 (\omega_X) = {\rm Pic}^0 (X)$ as well. We rephrase the final result as saying that
$$V^0 (\omega_X) = V^1 (\omega_X) \simeq {\rm Pic}^0 (C), \,\,\,\, 
V^2 (\omega_X) = \{0\}.$$
The preceding paragraph says that the exact same calculation must hold for a Fourier-Mukai partner $Y$. 

Let's now assume that $f$ is isotrivial. 
First note that for such an $X$ we have $q (X) = g + 1$, and in fact the Albanese variety of $X$ is an extension of abelian varieties 
$$ 1 \rightarrow F \rightarrow {\rm Alb} (X) \rightarrow J (C) \rightarrow 1$$
with $F$ an elliptic curve isogenous to the general fiber of $f$, though this will not play an explicit role in the calculation. Moreover, we have $\chi (\omega_X) = 0$ (see \cite{friedman} Ch.7, Lemma 14 and Corollary 17).

We now use a result of Beauville \cite{beauville2} \cite{beauville3}, characterizing the positive dimensional irreducible 
components of  $V^1 (\omega_X)$. Concretely, by \cite{beauville3} Corollaire 2.3, any such positive dimensional component would have to come either from a fiber space $h: X \rightarrow B$ over a curve of genus at least $2$, or from a fiber space 
$p: X \rightarrow F$ over an elliptic curve, with at least one multiple fiber.
Regarding the first type, the union of all such components
is shown in \emph{loc. cit.} to be equal to 
$${\rm Pic}^0 (X, h) : = {\rm Ker} \big({\rm Pic}^0 (X) \overset{i^*}{\longrightarrow} {\rm Pic}^0 (F)\big),$$
where $i^*$ is the restriction map to any smooth fiber $F$ of $h$. 
But since $f$ is an elliptic fibration, it is clear that there is exactly one such fiber space, namely $f$ itself 
(otherwise the elliptic fibers of any other fibration would have to dominate $C$, which is impossible). Therefore the
union of the components coming from fibrations over curves of genus at least $2$ is ${\rm Pic}^0 (X, f)$.
On the other hand,  for elliptic surfaces of the  type we are currently considering,  fibrations $p: X \rightarrow F$ over elliptic curves 
as described above do not exist. (Any such 
would have to come from a group action on a product between an elliptic curve $F^\prime$ and another of genus at least $2$, with the action on the elliptic component having no fixed points, therefore leading to an \'etale cover $F^\prime\rightarrow F$; in the language of \cite{beauville3}, we are saying that $\Gamma^0 (p) = \{0\}$.)

Using once more the deformation invariance of the Euler characteristic, we have 
$\chi (\omega_X\otimes \alpha) = 0$ for all $\alpha \in {\rm Pic}^0 (X)$, which gives
\begin{equation}\label{nonzero}
h^1 (X, \omega_X \otimes \alpha) = h^0 (X, \omega_X \otimes \alpha), \,\,\,\, {\rm for~all~} \alpha \neq \OO_X.
\end{equation}
This implies that ${\rm Pic}^0 (X, f)$ is also the union of all positive dimensional components of 
$V^0 (\omega_X)$.

We are left with considering nontrivial isolated points in $V^1(\omega_X)$ (or equivalently in $V^0 (\omega_X)$ by (\ref{nonzero})).
These can be shown not to exist by means of a different argument:
by a variant of the higher dimensional Castelnuovo-de Franchis inequality, see
\cite{LP} Remark 4.13, an isolated point $\alpha \neq 0$ in $V^1 (\omega_X)$ forces the 
inequality
$$\chi (\omega_X) \ge q(X) - 1 = g \ge 2,$$
which contradicts the fact that $\chi (\omega_X) = 0$.\footnote{As L. Lombardi points out, a variant of the derivative complex argument  in \cite{LP} leading to this inequality can also be used, as an alternative to Beauville's argument, in order to show that positive dimensional components not passing through the origin do not exist in the case of surfaces of maximal Albanese dimension with $\chi(\omega_X) = 0$.}

Putting everything together, we obtain
$$V^0 (\omega_X) = V^1 (\omega_X) =
 {\rm Pic}^0 (X, f), \,\,\,\, V^2 (\omega_X) = \{0\}.$$
Recall that a Fourier-Mukai partner of $X$ must be an elliptic fibration of the same type over $C$.
Now one of the main results of \cite{pham}, Theorem 5.2.7,  says that for derived equivalent elliptic 
fibrations $f: X \rightarrow C$ and $h: Y \rightarrow C$ which are isotrivial with only multiple fibers, one has
 $${\rm Pic}^0 (X, f) \simeq {\rm Pic}^0 (Y,h ),\footnote{This also follows from \cite{lombardi}, via Theorem \ref{luigi} below.}$$
 which allows us to conclude that $V^i (\omega_X)$ and
$V^i (\omega_Y)$ are isomorphic.
\end{proof}

\section{Further evidence and applications}

\noindent
{\bf Progress.}
Progress towards the conjectures in \S1 has been made by Lombardi \cite{lombardi}.
The crucial point is to come up with an explicit mapping realizing the potential isomorphisms in 
Conjecture \ref{main}. This is done by means of the \emph{Rouquier isomorphism}; namely, 
given a Fourier-Mukai equivalence $\R\Phi_{\E} : \D(X) \rightarrow \D(Y)$ induced by an object
$\E \in  \D(X\times Y)$, Rouquier \cite{rouquier} Th\'eor\`eme 4.18 shows that there is an induced isomorphism of algebraic groups
$$F : {\rm Aut}^0 (X) \times {\rm Pic}^0 (X) \longrightarrow {\rm Aut}^0 (Y) \times {\rm Pic}^0 (Y)$$
given by a concrete formula involving $\E$ (usually mixing the two factors), \cite{PS} Lemma 3.1. 
A key result in
\cite{lombardi} is that if $\alpha \in V^0 (\omega_X)$ and 
$$F({\rm id}_X, \alpha) = ( \varphi, \beta),$$
then in fact $\varphi = {\rm id}_Y$, $\beta \in V^0 (\omega_Y)$, and moreover 
\begin{equation}\label{precise}
H^0 (X, \omega_X\otimes \alpha) \simeq H^0 (Y, \omega_Y \otimes \beta).
\end{equation}
One of the main tools used there 
is the derived invariance of a generalization of Hochschild homology taking into account the 
Rouquier isomorphism. This implies the  invariance of $V^0$, while further work using a variant of
the Hochschild-Kostant-Rosenberg isomorphism gives the following, again the isomorphisms being 
induced by the Rouquier mapping.

\begin{theorem}[Lombardi \cite{lombardi}]\label{luigi}
Let $X$ and $Y$ be smooth projective varieties with $\D(X) \simeq \D(Y)$. Then:

\noindent
(i) $V^0 (\omega_X) \simeq V^0 (\omega_Y)$. 

\noindent
(ii) $V^1 (\omega_X) \cap V^0 (\omega_X) \simeq V^1 (\omega_Y) \cap V^0 (\omega_Y)$. 

\noindent
(iii) $V^1 (\omega_X)_0 \simeq V^1 (\omega_Y)_0$. 
\end{theorem}

This result recovers Theorem \ref{surfaces} in a more formal way.
In the case when $\dim X = \dim Y = 3$, with extra work one shows that this has the following consequences, 
verifying or getting close to verifying the various conjectures:

\noindent
$\bullet$~ \,\,~Variant \ref{var1} holds.

\noindent
$\bullet$~\,\,~For any $i$, $V^i (\omega_X)$ is positive dimensional if and only if $V^i (\omega_Y)$ is positive 
dimensional, and of the same dimension. Therefore Variant \ref{var2} holds, except for the possible case where for 
some $i >0$, $V^i (\omega_X)$ is finite, while $V^i (\omega_Y) = \emptyset$. This last case can possibly happen 
only when $q (X) = 1$.

\noindent
$\bullet$~\,\,~Conjecture \ref{main} is true when:
\begin{enumerate}
\item $X$ is of maximal Albanese dimension (i.e. the Albanese map of $X$ is generically finite onto its image).
\item $V^0 (\omega_X) = {\rm Pic}^0 (X)$ -- for instance, by \cite{PP} Theorem E, this condition holds whenever 
the Albanese image $a(X)$ is not fibered in subtori of ${\rm Alb}(X)$, and $V^0 (\omega_X) \neq \emptyset$. 
\item ${\rm Aut}^0 (X)$ is affine --  this holds for 
varieties which are not isotrivially fibered over a positive dimensional abelian variety (see \cite{brion} p.2 and \S3), for instance again when the Albanese image is not fibered in subtori of ${\rm Alb}(X)$ according to a theorem of Nishi (cf. \cite{matsumura} Theorem 2).
\end{enumerate}

These conditions together impose very strong restrictions on the threefolds for which the conjecture is 
not yet known.
Note finally that in \cite{lombardi} there are further extensions involving cohomological support loci for $\omega_X^{\otimes m}$ with 
$m \ge 2$, and for $\Omega_X^p$ with $p < \dim X$.

\noindent
{\bf Some first applications.}
Let $X$ be a smooth projective complex variety of dimension $d$, and let $a: X \rightarrow A = {\rm Alb}(X)$ 
be the Albanese map of $X$.  
A first consequence of the weakest version of the conjectures would be the derived invariance of the Albanese 
dimension $\dim a(X)$.

\begin{corollary}[\textbf{assuming Variant \ref{var3}}]\label{albanese_dimension}
If $X$ and $Y$ are smooth projective complex varieties with $\D(X) \simeq \D(Y)$, then
$$\dim a(X) = \dim a(Y).$$
\end{corollary}

This follows from the fact that, according to \cite{LP} Remark 2.4,  the Albanese dimension can be computed from the dimension 
of the cohomological support loci around the origin, according to the formula
$$\dim a(X) = \underset{i = 0, \ldots, d}{{\rm min}} \{ d - i + {\rm codim}~V^i (\omega_X)_0 \}.$$
Note that Lombardi \cite{lombardi} is in fact able to prove Corollary \ref{albanese_dimension} 
when $\kappa(X) \ge 0$ by relying on different tools from birational geometry.
The only progress when $\kappa(X) = -\infty$, namely a solution for surfaces and threefolds 
that can  also be found in \emph{loc. cit.}, involves the approach described here. 

Another numerical application involves the holomorphic Euler characteristic. While the individual Hodge numbers are not yet 
known to be preserved by derived equivalence, the Euler characteristic can be attacked in some cases by using generic vanishing theory and 
the derived invariance of $V^0 (\omega_X)$ established in Theorem \ref{luigi}.

\begin{corollary}[\emph{of Theorem \ref{luigi}}, \cite{lombardi}]
If $X$ and $Y$ are smooth projective complex varieties with $\D(X) \simeq \D(Y)$, and $X$ is of maximal Albanese dimension, 
then $\chi (\omega_X) = \chi (\omega_Y)$.
\end{corollary}

This follows from the fact that, according to (\ref{gv}), for generic $\alpha \in {\rm Pic}^0 (X)$ one has
$$\chi (\omega_X) = \chi (\omega_X \otimes \alpha) = h^0 (X, \omega_X \otimes \alpha),$$
combined with (\ref{precise}). The argument is extended in \cite{lombardi} to other cases as well.
Going back to Hodge numbers, this implies for instance that if $X$ and $Y$ are derived 
equivalent $4$-folds of maximal Albanese dimension, then 
$$h^{0,2} (X) = h^{0,2} (Y),\footnote{This also automatically implies $h^{1,3}(X) = h^{1,3}(Y)$ by the invariance of Hochschild homology.}$$
since in the case of $4$-folds all the other $h^{0,q}$ Hodge numbers are known to be preserved.

Perhaps the main point in this picture is the fact that the positive dimensional components of the cohomological support
loci $V^i (\omega_X)$ reflect the nontrivial fibrations of $X$ over irregular varieties. Therefore, roughly speaking, the key geometric significance of Conjecture \ref{main} is that derived equivalent varieties should have the same type of fibrations over lower dimensional irregular varieties, thus allowing 
for more geometric tools in the study of Fourier-Mukai partners.
One version of this principle can be stated as follows:

\begin{corollary}[\textbf{assuming Variant \ref{var1}}]\label{fibrations}
Let $X$ and $Y$ be smooth projective varieties such that $\D(X) \simeq \D(Y)$. Fix an integer $m>0$,
 and assume that $X$ admits a morphism $f: X \rightarrow Z$ with connected fibers, onto a normal irregular variety of dimension $m$ whose Albanese map is not surjective. 
 Then $Y$ admits a morphism
$h: Y \rightarrow W$ with connected fibers, onto a positive dimensional normal irregular variety of dimension $\le m$. Moreover, if $m= 1$, then 
$W$ can also be taken to be a curve of genus at least $2$.
\end{corollary}

This is due to the fact that, by the degeneration of the Leray spectral sequence for $\R f_* \omega_X$ 
due to Koll\'ar \cite{kollar1}, one has 
$$f^* V^0 (\omega_Z) \subset V^{n - m} (\omega_X),$$ 
where $n = \dim X = \dim Y$.
Now in \cite{EL} Proposition 2.2 it is shown that if $0$ is an isolated point in $V^0 (\omega_Z)$, then
the Albanese map of $Z$ must be surjective. Thus the hypothesis implies that we obtain a positive 
dimensional component in $V^{n - m} (\omega_X)_0$, hence by Variant \ref{var1} also in 
$V^{n - m} (\omega_Y)_0$. Going in reverse, 
recall now from \S1 that, according to one of the main results of 
\cite{GL2}, a positive dimensional component of $V^{n- m} (\omega_Y)$ produces a fiber space 
$h: Y\rightarrow W$,  
with $W$ a positive dimensional normal irregular variety (with generically finite Albanese map) and $\dim W \le m$. 
The slightly stronger statement in the case of fibrations over curves follows from the precise description of the positive 
dimensional components of $V^{n-1} (\omega_X)$ given in \cite{beauville3}.

I suspect that one should be able to remove the non-surjective Albanese map hypothesis (in other
words allow maps onto abelian varieties), but this must go beyond the methods described here.

\noindent
{\bf Acknowledgements.} 
I thank L. Lombardi and C. Schnell for numerous useful conversations on the topic presented here.

\providecommand{\bysame}{\leavevmode\hbox
to3em{\hrulefill}\thinspace}

\end{document}